\documentclass{article}

\usepackage{amsmath,amscd,amsthm,amssymb,amsxtra,latexsym,epsfig,epic,graphics}
\usepackage[matrix,arrow,curve]{xy}
\usepackage{graphicx}

\title{Betti Numbers of Syzygies and Cohomology of Coherent Sheaves}
\author{David Eisenbud and Frank-Olaf Schreyer}

\date{}
\begin{document}
\maketitle

\newtheorem{theorem}{Theorem}[section]
\newtheorem{prop}[theorem]{Proposition}
\newtheorem{lemma}[theorem]{Lemma}
\newtheorem{cor}{Corollary}
\newtheorem{definition}{Definition}

\begin{abstract} The Betti numbers of a graded module over the polynomial ring form a table of numerical invariants that refines the Hilbert polynomial. A sequence of papers sparked by conjectures of Boij and S\"oderberg have led to the characterization of the possible Betti tables up to rational multiples---that is, to the rational cone generated by the Betti tables. We will summarize this work by describing the cone and the closely related cone of cohomology tables of vector bundles on projective space, and we will give new, simpler proofs of some of the main results. We also explain some of the applications of the theory, including the one that originally motivated the conjectures of Boij and S\"oderberg, a proof of the Multiplicity Conjecture of Herzog, Huneke and Srinivasan. 
\end{abstract}



%

\def\sL{{\mathcal L}}

\def\sU{{\mathcal U}}
\def\sG{{\mathcal G}}
\def\sF{{\mathcal F}}
\def\sE{{\mathcal E}}
\def\sO{{\mathcal O}}
\def\sM{{\mathcal M}}
\def\coker{{\rm coker}}
\def\bF{{\bf F}}
\def\QQ{{\mathbb Q}}
\def\RR{{\mathbb R}}
\def\PP{{\mathbb P}}
\def\ZZ{{\mathbb Z}}
\def\lTo{{\leftarrow}}
\def\codim{{\hbox{codim}\,}}
\def\mult{{{\rm mult}\,}}

\section{Introduction}

Hilbert's Syzygy theorem states that
every finitely generated graded module over a polynomial ring $S=K[x_1,\ldots,x_n]$
has a finite free  resolution of length at most $n$. Hilbert's motivation was to 
 show that the \emph{Hilbert function} $h_{M}(k):= \dim_{K}M_{k}$ is given by a polynomial $p_{M}(k)$ for large $k$, as follows:  for a graded module
 $M= \oplus_k M_k$ over the standard graded polynomial ring $S$ consider a finite free graded resolution, that is, an exact sequence
 $$\bF: \quad (0\lTo M  \lTo )F_0 
  \lTo 
  F_1 \lTo \ldots \lTo F_r \lTo 0$$
 with $F_i=\oplus_j S(-j)^{\beta_{i,j}}$ and $S(-j)$ the free cyclic $S$ module with generator in degree $j$. The numbers $\beta_{i,j}$ are called
 the \emph{graded Betti numbers} of the resolution. From the exactness of the resolution  Hilbert deduced the formula
$$ 
h_M(k)=\dim_K M_k = \sum_{i=0}^r (-1)^i \sum_j \beta_{i,j} {k+n-j \choose n} 
$$
for the Hilbert function $h_M$. Since the combinatorial binomial coefficient ${m\choose n}$ agrees with the polynomial
$$
{m(m-1)\cdots (m-n+1)\over n!} \in \QQ[m]
$$ 
 when $m$ is large enough,
the same formula defines  the Hilbert polynomial $p_M(k) \in \QQ[k]$. 
The Hilbert polynomial $p_M(k)$ captures the most important properties of $M$. For example,
$\deg p_M= \dim M$, and the leading coefficient of $p_M$ times ${(\dim M)!}$ is the  the multiplicity  of $M$.

A minimal free resolution of a module $M$ is a free resolution $\bF$ such that  no proper summand of $F_{i+1}$ maps surjectively onto the kernel
$\ker(F_{i} \to F_{i-1})$. It is determined up to isomorphism by $M$ (see for example \cite{E1}), and thus the \emph{graded Betti numbers} $\beta_{i,j}=\beta_{i,j}(M)$ of a minimal free resolution are invariants of $M$.
We collect the Betti numbers of $M$ as usual in a \emph{Betti table} 
$
\beta(M)=(\beta_{i,j}(M)) 
$. 
Hence, $\beta(M)$ is a numerical invariant that refines the Hilbert polynomial. 

There are many papers
whose goal is to describe this invariant and its possible values in special cases. 
In  2006 Mats Boij and Jonas S\"oderberg suggested a relaxation of this problem that
opened the door to a radically different approach.
The set of 
Betti tables form a semigroup, since the direct sum of modules corresponds
to the addition of Betti tables. Allowing multiplication by positive rational
numbers instead of just positive integers, we get a rational convex
cone, the cone of Betti tables.
Boij and S\"oderberg conjectured that the \emph{extremal rays} of this cone are spanned by Betti 
tables of so called \emph{pure resolutions}, described below.

In \cite{ES1} we showed that these  conjectures were true. Besides proving 
the existence of pure resolutions this involved finding the equations of the facets of the cone of Betti tables.
To describe how this was done, we must introduce another invariant of a finitely generated graded module
$M$ that also refines the Hilbert polynomial.
Consider the coherent sheaf $\sE$ on $\PP^{n-1}$ represented by $M$. 
The value of the Hilbert polynomial  $p_{M}(k)$ coincides with  the Euler characteristic of the twisted sheaf: $\chi \sE(k)$.
Since the Euler characteristic is the alternating sum of the dimensions of the cohomology groups, we can also think of the
 \emph{cohomology table} $\gamma(\sE)=(h^j \sE(k))$ of $\sE$ as a refinement of the Hilbert polynomial. And, just as with
 Betti tables, it is natural to consider the convex rational cone generated by the cohomology tables.

The key idea in our proof of the Boij-S\"oderberg conjecture was to show that the facets of the cone of Betti tables of finite length modules come from the extremal rays
in the closed cone of cohomology tables of vector bundles on projective space. We identified these extremal rays
and showed that vector bundles
with such extremal cohomology tables exist, so that the cone of cohomology tables, like the cone of Betti tables, is closed. 
Using our results, Boij and S\"oderberg extended the theory to arbitrary modules \cite{BS2}. On the cohomology side, we generalized the results to arbitrary coherent sheaves \cite{ES2}. 

A flurry
of other papers and preprints including  \cite{EFW},  \cite{ESW}, \cite{Erman1}
\cite{Erman2} and \cite{EES} have added to the basic picture and its applications. 
In this note we give a new and simpler proof  of our main result on the cone of Betti tables, and we give a simpler treatment of
the theorem of Boij and S\"oderberg on Betti tables of arbitrary finitely generated modules. We also explain some applications,
and survey what is known in some other cases. 

\section{Betti Tables}

As above, let  $S=K[x_1,\dots,x_n]$ be a polynomial ring over a field $K$, graded with each $x_i$ of degree 1. For a finitely generated graded $S$-module $M$ 
we may regard its
Betti table as an integral point in an infinite dimensional $\QQ$-vector space:
$$\beta(M)=(\beta_{i,j}(M)) \in \bigoplus_{i=0}^n \bigoplus_{j\in \ZZ} \QQ.$$

The cone generated by finite positive rational linear combinations of Betti  tables is called the \emph{Boij-S\"oderberg cone}.
Our main result on Betti tables is a description of this cone in terms of its extremal rays.

\begin{definition}
A 
finitely generated graded $S$-module $M$ is called \emph{pure} of 
type $d:=(d_0,\dots, d_r)$ if
\begin{enumerate}
\item $M$ is Cohen-Macaulay of codimension $r$; that is, $F_i=0$ for $i>r$ and $\dim M=n-r$ 
\item in a minimal
free resolution of $M$ as above, the free module $F_i$ is generated by elements of degree $d_i$;
that is, $\beta_{i,j}=0$ when $j\neq d_i$.
\end{enumerate}
\end{definition}

\begin{prop} If $M$ is a pure module of type $d$, then $\beta(M)$ is determined up to a rational factor by $d$. 
Further, $\beta(M)$ spans an extremal ray in the Boij-S\"oderberg cone of Betti tables.
\end{prop} 

\begin{proof} Suppose $\beta(M)= \sum_{\ell=1}^N q_\ell \beta(M_\ell)$ with rational numbers $q_\ell >0$. We have to prove that the Betti tables $\beta(M_\ell)$ all lie in the same ray as $\beta(M)$. Each of the modules $M_\ell$ has $\dim M_\ell \le \dim M$, because otherwise the Hilbert polynomial of $M$ would have larger degree. Since by the Auslander-Buchsbaum-Serre formula,
the length of a free resolution is at least the codimension, the equality of Betti tables implies that
each $M_\ell$ is Cohen-Macaulay with the same codimension $r$ as $M$ and that each $M_\ell$ is pure with the same type $(d_1,\ldots,d_r)$ as $M$. The proof that each $\beta(M_\ell)$ lies in the same ray as
$\beta(M)$ follows by an argument of Herzog and K\"uhl \cite{HK}: Consider the Hilbert series
of $M$ defined as
$$H_M(t)= \sum_{d\in \ZZ} \dim M_d t^d = \frac{\sum_{i,j} (-1)^i \beta_{i,j} t^j}{(1-t)^n} \in \QQ[[t]][t^{-1}].$$
This rational function has a pole of order $\dim M$ at $t=1$, or equivalently, the
numerator $\sum_{i,j} (-1)^i \beta_{i,j} t^j$ has a zero of order $\codim M$. In case of a pure module the
numerator simplifies to
$$\sum_{i=0}^r (-1)^i \beta_{i,d_i} t^{d_i}.$$
Hence, the $r+1$ numbers $\beta_{0,d_0}, \ldots, \beta_{r,d_r}$ satisfy a system of $r$ linear equations
$$\sum_{i=0}^r \beta_{i,d_i} d_i^s=0 \hbox{ for } s=0,\ldots r-1$$
of Vandermond type. Hence by Cramer's rule,
$$\beta_{i,d_i}= q \prod_{t>s,\; t,s\not=i} (d_t-d_s )$$ 
are determined by the type up to a common rational factor $q$.

\end{proof}

\noindent
With $\beta(d)$ we denote the rational table on the ray of type $d$ normalized such that
$$\beta_{i,d_i}(d) =\prod_{j\not=i} \frac{1}{| d_j-d_i |}.$$

To formulate our main result one more preparation is necessary:
we order the strictly increasing sequences $d$ as follows:
$$
d=(d_0,\dots, d_r)\leq d'=(d_0',\dots,d_{r'}')
$$
if $r\geq r'$ and $d_i \leq d'_i$ for $i=1,\dots,r'$. One can think of this as the termwise order if one
simply extends each sequence $d=(d_0,\dots, d_r)$ to \
$(d_0,\dots, d_r, \infty, \infty,\dots)$.

We can now state the main result of the theory concerning the cone of Betti tables:

\begin{theorem}[\cite{BS2,ES1}] \label{bettimain} Let $S=k[x_1,\dots,x_n]$ be as above. 
\begin{enumerate}
\item For every strictly increasing sequence of integers $d=(d_0, \dots, d_r)$ with $r\leq n$, there
exist pure  $S$-modules of type $d$.
\item The Betti table of any finitely generated graded $S$-module may be written uniquely as
a positive rational linear combination of the Betti tables of a set of pure  modules
whose types form a totally ordered sequence.

\end{enumerate}
\end{theorem}

The second statement of the theorem has two nice interpretations that may help to clarify its meaning.
First,  geometrically, it really says that the cone of Betti tables is
a \emph{simplicial fan}, that is, it is the union of simplicial cones, meeting along facets. The maximal 
simplicial cones in the fan
correspond to maximal chains (totally ordered subsets) in the partially ordered set of degree sequences: the simplicial
cone is the set of finite positive rational combinations of Betti tables whose degree sequences 
lie in the chain. 
These simplices and cones are thus infinite ascending unions of finite-dimensional cones, corresponding
to Betti tables with finite support, that is, resolutions
where the free modules are generated in a given bounded range of degrees.

Second, algorithmically, the theorem implies that there is a greedy algorithm that gives
the decomposition. Rather than trying to specify this formally, we give an example with $n=3$.
To describe it compactly, we will write the Betti table of a module $M$ as an array whose entries in the $i$-th
column are the $\beta_{i,j}$---that is, the $i$-th column corresponds to the free module $F_i$. For
reasons of efficiency and tradition, we put $\beta_{i,j}$ in the $(j-i)$-th row.

Consider the $K[x,y,z]$-module $M=S/(x^2, xy, xz^2)$. The minimal free resolution of
$M$ has the form
$$
S \lTo S(-2)^2\oplus S(-3) \lTo S(-3)\oplus S(-4)^2 \lTo S(-5) \lTo 0
$$
and is represented by an array
$$
\beta(M)=\begin{pmatrix}
1\\
& 2&1\\
&1&2&1
\end{pmatrix}
$$
where all the entries not shown are equal to zero.

To write this as a positive rational linear combination of pure diagrams, we first consider the ``top row'', corresponding
to the generators of lowest degree in the free modules of the resolution. These are in the positions 
$$
\begin{pmatrix}
*\\
& *&*\\
&&&*
\end{pmatrix}
$$
corresponding to the degree sequence $(0,2,3,5)$. There is in fact a pure module $M_1= S/I_1$ with resolution
$$
\beta(M_1) = \begin{pmatrix}
1\\
& 5&5\\
&&&1
\end{pmatrix}.
$$
The greedy algorithm now instructs us to subtract \emph{the largest possible $q_1$ 
that will leave the resulting table $\beta(M)-q_1\beta(M_1)$ having only non-negative terms.} We see at
once that $q_1= 1/5$.

We now repeat this process starting from $\beta(M)-q_1\beta(M_1)$; the theorem guarantees that there
will always be a pure resolution whose degree sequence matches the top row of the successive remainders.
In this case we arrive at the expression 
\begin{eqnarray*}
\beta(M)=\begin{pmatrix}
1\\
& 2&1\\
&1&2&1
\end{pmatrix}=&
{1/5}
\begin{pmatrix}
1\\
& 5&5\\
&&&1
\end{pmatrix}
+
{1/10}
\begin{pmatrix}
3\\
& 10&\\
&&15&8
\end{pmatrix} \cr
\quad &\vspace{0.5cm}\cr
&+
{1/6}
\begin{pmatrix}
1\\
& &\\
&4&3&\,
\end{pmatrix}
+
{1/3}
\begin{pmatrix}
1\\
& \\
&1&&\quad&
\end{pmatrix}.\cr
\end{eqnarray*}
\normalsize
All the fractions and tables that occur are of course invariants---apparently new invariants---of $M$.

\section{Facets of the Cone and Cohomology Tables}

We next focus on the facets of the cone of Betti tables (compare with \cite{BS1,BS2}).
Consider a finite chain of degree sequences. Since the rays corresponding to the degree sequences are linearly independent, these rays generate  a simplicial cone.
A facet (maximal face) of this cone is generated by all but one of the rays in our chain. If this ray corresponds to the degree sequence $b$, then we may assume that $b$ is neither the largest nor the smallest degree sequence by replacing our chain with a longer chain of degree sequences. Consider the  degree sequences $a>b$ and $c<b$ immediately above and below in our chain. 
By inserting a finite number of degree sequences between $a$ and $b$ if necessary, we can achieve that $a$ and $b$ differ in at most one position. Similarly, we can achieve that $b$ and $c$ differ in at most one position. Then the facet of this simplicial cone obtained by deleting $b$ is  an
\emph{outer face}---that is, it will lie on the boundary of the cone of Betti tables---if
 either $a$ and $c$ differ in precisely one position $\tau$, or
$a$ and $c$ differ in precisely two consecutive positions $\tau$ and $\tau+1$  and
$a_\tau \ge c_{\tau+1}$, compare \cite{BS1}, Proposition 2.2.
Indeed, suppose  that $a$ and $c$ differ in the position $\tau$ and $k$ with $k>\tau$ and, moreover, $k > \tau+1$
or $k=\tau+1$ and $a_{\tau}<c_{\tau+1}$.
Then both sequences 
$b=(\ldots,a_{\tau}, \ldots,  c_{k}, \ldots)$ and $b'=( \ldots,c_{\tau},\ldots, a_{k}, \ldots)$ are increasing, hence valid degree sequences between with $a$ and $c$. For $a_k<\infty$ we have the numerical  identity
$$(a_k-a_\tau)\beta(a)+(c_k-c_\tau)\beta(c)=(c_k-a_\tau)\beta(b)+(a_k-c_\tau)\beta(b'),$$
which becomes  $\beta(a)+(c_k-c_\tau)\beta(c)=(c_k-a\tau)\beta(b)+\beta(b')$, in case $a_k=\infty$.
Hence, once we know that pure modules for  arbitrary degree sequences exist, we can deduce that the facet of the simplicial  cone obtained by dropping $b$ (or $b'$) lies in the interior of the cone of Betti tables.

Hence, apart from the existence of pure resolution, we have to show that every potential outer face as above is indeed an outer face.
A typical example which could lead to an outer face is  the chain
$$a=(0,3,4) >b=(0,2,4) > c=(0,1,4)$$
for the case that $a$ and $c$ differ in only one position and
$$a=(0,2,3,4)>b=(0,1,3,4) >c=(0,1,2,4)$$
in case $a$ and $c$ differ in two positions. In the first case, the linear function $\beta \mapsto \beta_{\tau,b_\tau}$ is positive on the ray corresponding to $b$ and vanishes on all other rays in any simplex corresponding to a chain of degree sequence containing $a, b, c$. Clearly, this functional is also non-negative on Betti tables of arbitrary module. So these are indeed outer faces.

The second case is more complicated. 
We start by  replacing $a,b$ and $c$ by
$(\ldots,a_{\tau-1},c_{\tau+1},c_{\tau+1}+1,\ldots), (\ldots,a_{\tau-1},c_{\tau+1}-1,c_{\tau+1}+1,\ldots)$
and
$(\ldots,a_{\tau-1},c_{\tau+1}-1,c_{\tau+1},\ldots)$. 
For example, we will replace the triple $$(0, \infty,\infty) > (0,1,\infty) >(0,1,3)$$ by $$(0,3,4)>(0,2,4)>(0,2,3).$$
We will see that the equation, which  we will derive below in this new situation, works for the face obtained by deleting the original $b$ as well.

Now take a complete chain extending $a>b>c$.
We can compute the coefficient $\delta_{i,j}$ of  a functional 
$\delta: \beta \mapsto \sum_{ij} \delta_{i,j} \beta_{i,j}$ vanishing on the facet opposite to $b$ recursively.
We start by taking $\delta_{i,a_i}=0$ and work our way up and down in the chain. The condition
$\delta(\beta(c))=0$ 
determines the ratio of $\delta_{\tau,c_\tau}$ while $\delta_{\tau+1,c_{\tau+1}}$ and $\delta(\beta(b))>0$ determines the sign. Moving one step down from $c$ in the chain, determines one more coefficient of $\delta$.
In the example above we can look at the degree sequence
$$(0,1,3,4)>(0,1,2,4)>(0,1,2,3)=(0,1,2,3,\infty)> \ldots >(0,1,2,3,6)>$$
$$(0,1,2,3,5)>(0,1,2,3,4)>(-1,1,2,3,4)>\ldots$$
whose Betti tables are
$$
\begin{pmatrix}
 & & & &\\
2 & 4 &  & &\\
  &  &  4   &2&\;\\
     & & & &\\
\end{pmatrix} 
\begin{pmatrix}
 & & & &\\
3 &{\bf 8} & {\bf 6} & &\\
  &  & &     1&\;\\
     & & & &\\
\end{pmatrix}
\begin{pmatrix}
 & & & &\\
1 & 3 & 3 & {\bf 1}&\\
  &  & &   &\;  \\
     & & & &\\
\end{pmatrix}
\cdot \cdot \cdot
\begin{pmatrix}
\;&  &  &  &\\
   10&  36& 45 & 20 &\\ 
        & & & & \\
   & & & & {\bf 1}\\

\end{pmatrix}
$$
$$
\begin{pmatrix}
\;&  &  &  &\\
   4&  15& 20 & 10 &\\
   & & & & {\bf 1}\\
      & & & & \quad\\
\end{pmatrix}
\begin{pmatrix}
\;&  &  &  &\\
   1&  4& 6 &  4 & {\bf 1}\\
   & & & & \quad\\
      & & & & \quad\\
\end{pmatrix}
\begin{pmatrix}
{\bf 1} &  &  &  &\\
  &  10& 20 &  15 & 4\\
   & & & & \quad\\
      & & & & \quad\\
\end{pmatrix}
\ldots
$$
and obtain
$$\delta=(\delta_{i,j})=
\begin{pmatrix}
\vdots    & \vdots & \vdots & \vdots & \vdots \\
21 &-12 & 5 & 0 &-3\\
12 & -5 & 0 & \;3 &-4\\
5 & 0 & -3 & 4 &-3\\
 \bf 0 & 3 & -4& 3 &0\\
 0 & \bf 0& \bf 0&   \bf 0 &5 \\
  0 & 0& 0&   0  &12\\
\vdots    & \vdots & \vdots & \vdots & \vdots  \\
\end{pmatrix} 
$$
Of course, moving up in degree from $a$ will always yield zero coefficients. Note, that  the results of these computations apparently do not depend on the specific choice of the complete chain extending $a>b>c$.

To prove Theorem \ref{bettimain},
we have to show that each such $\delta$  is nonnegative on the Betti table $\beta(M)$ of an arbitrary module. Our key observation is that the numbers appearing are dimensions of cohomology groups of what we call supernatural vector bundles on $\PP^{r-1}$.

\begin{definition} A vector bundle $\mathcal E$ on $\mathbb P^m$ has \emph{natural cohomology} \cite{VB} if for each $k$ at most one of the groups
$$H^i(\mathcal E(k)) \not=0.$$
It has \emph{supernatural cohomology} if in addition the Hilbert polynomial
$$\chi(\mathcal E(k)) = \frac {\hbox{ rank } \mathcal E}{m!} \prod_{j=1}^m (k-z_j)$$
has m distinct integral roots $z_1 > z_2 >\ldots  > z_m$.
\end{definition}

\noindent Note that a supernatural vector bundle $\sE$ has non-vanishing cohomology  in the following range
(see \cite{ES1}):
$$
\begin{cases} H^0(\sE(k)) \not=0 \\
H^i(\sE(k)) \not=0 \\
H^m(\sE(k))\not=0 \\
\end{cases}
\hbox{if and only if }
\begin{cases}
\qquad k >z_1\cr
z_{i}>k > z_{i+1} \cr
 z_m > k\cr
\end{cases}.$$

\noindent
For a coherent sheaf $\mathcal E$ on $\mathbb P^m$ we denote by
$$\gamma(\mathcal E)=(\gamma_{j,k}) \in \bigoplus_{j=0}^m \prod_{k\in\ZZ}\QQ \hbox{ with } \gamma_{j,k}=h^j(\mathcal E(k))$$ its cohomology table. Analogous  to the theorem on free resolutions we have

\begin{theorem}[\cite{ES1}]\label{mainbundles} The extremal rays of the  rational cone of cohomology tables of vector bundles on $\PP^{m}$ are generated by
 cohomology tables of supernatural vector bundles. 
 
 More precisely:
 Every cohomology table of a vector bundle is a unique positive rational combination of cohomology tables of supernatural vector bundles, whose root sequences form a chain.
\end{theorem}

\noindent
Here we order the root sequences component wise. \medskip

The crucial new concept is the following pairing between Betti tables of modules and cohomology tables of coherent sheaves.
We define $\langle \beta, \gamma \rangle$ for a  Betti table $\beta=(\beta_{i,k})$ and a cohomology table $\gamma=(\gamma_{j,k})$ by
$$\langle \beta, \gamma \rangle = \sum_{i \ge j} (-1)^{i-j} \sum_k \beta_{i,k}\gamma_{j,-k}$$

\begin{theorem}[Positivity 1,\cite{ES1,ES2}]\label{pos1}
For $\bF$ any free resolution of a finitely generated graded $K[x_0,\ldots,x_m]$-module $M$
and $\mathcal E$ any coherent sheaf on $\mathbb P^m$, we have
$$\langle \beta(\bF), \gamma(\mathcal E) \rangle \ge 0.$$
Moreover, if $M$ has finite length and
$H^{i+1}(\widetilde F_{i} \otimes \sE)=0$ for all $i \ge 0,$
then 
$$\langle \beta(\bF), \gamma(\mathcal E) \rangle = 0.$$

\end{theorem}

\noindent
Note that if $\widetilde F_i= \oplus_{j\in \ZZ} \sO(-j)^{\beta_{i,j}}$ then
$$\langle \beta(\bF), \gamma(\mathcal E) \rangle=\sum_{i \ge j} (-1)^{i-j} h^j(\widetilde F_i\otimes \sE)$$

\begin{proof}
We first treat the case where $\sE$ is a vector bundle. In this case we
have an exact complex 
$$
0 \lTo \sM_0 \lTo \widetilde F_0\otimes \sE \to \widetilde F_1\otimes \sE \lTo  \ldots \lTo \widetilde F_r\otimes \sE \lTo 0
$$
with $\sM_0= \widetilde M \otimes \sE$. Breaking it up in short exact sequences
$$
\begin{matrix} 
0 &\lTo& \sM_0 &\lTo &\widetilde F_0\otimes \sE& \lTo& \sM_1 &\lTo &0\cr
\cr
0 &\lTo& \sM_1 &\lTo &\widetilde F_1\otimes \sE& \lTo& \sM_2 &\lTo &0\cr
\cr
0&\lTo& \sM_2& \lTo& \widetilde F_2\otimes \sE& \lTo& \sM_3 &\lTo& 0 \cr
&&&&\vdots
\end{matrix}$$
we get the desired functional by taking the alternating sum of the  Euler characteristics of initial parts of the corresponding long exact sequences in cohomology:
$$
\begin{matrix} 
 && &H^0(\widetilde F_0\otimes \sE)&\lTo& H^0(\sM_1) &\lTo& 0 \cr
\cr
&&&H^1(\widetilde F_1\otimes \sE)& \lTo& H^1(\sM_2) & \lTo\cr
 &H^0(\sM_1)& \lTo &H^0(\widetilde F_1\otimes \sE)&\lTo& H^0(\sM_2) &\lTo& 0 \cr
\cr
&&&H^2(\widetilde F_2\otimes \sE)& \lTo& H^2(\sM_3) & \lTo \cr
&H^1(\sM_2)&\lTo&H^1(\widetilde F_2\otimes \sE)& \lTo& H^1(\sM_3) & \lTo\cr
 &H^0(\sM_2)& \lTo &H^0(\widetilde F_2\otimes \sE)&\lTo& H^0(\sM_3) &\lTo& 0 \cr
&&&\vdots \cr
\end{matrix} 
$$
Hence, $\langle \beta(\bF), \gamma(\mathcal E)\rangle=\sum_{j=0}^m \dim \coker H^j(\sM_{j+1}) \to H^j(\widetilde F_j \otimes \sE)) \ge 0.$

In the general case, where $\sE$ is not necessarily locally free, the complex at the beginning of the proof may not be exact.
However, we note that what we need to prove depends only on the cohomology
table of $\sE$, not on the sheaf itself. Hence, we can replace $\sE$ with a translate $g^* \sE$ for any  $g \in PGL(m+1)$. When
$g$ is a general element,  \cite{MS} shows that the sheaves $Tor_{i}(\tilde M, \sE) = 0$ for $i>0$; that is, the complex
$$
0 \lTo \tilde M \otimes g^{*}\sE \lTo \widetilde F_0\otimes g^{*}\sE \to \widetilde F_1\otimes \sE \lTo  \ldots \lTo \widetilde F_r\otimes g^{*}\sE \lTo 0
$$
is exact, and the same argument applies.

For the vanishing statement, we note that in this case $\widetilde \bF \otimes \sE$ is exact as well, and 
$\sM_0=0$. By induction we obtain $H^i(\sM_i)=0$ from $H^i( \widetilde F_{i-1} \otimes \sE)=0$, and all the Euler characteristics are in fact zero.
 \end{proof}

\noindent
The facet equation in the example above is obtained from the vector bundle $\mathcal E$ on $\mathbb P^2\stackrel{\iota}{\hookrightarrow} \PP^3$, that is the kernel of a general map $\mathcal O_{\PP^2}^5(-1) \to \mathcal O_{\PP^2}^3$. The coefficients of the functional
 $ \langle - , \gamma(\iota_*\mathcal E)\rangle$ are 
$$
\begin{pmatrix}
\vdots    & \vdots & \vdots & \vdots & \vdots\\
21 &-12 & 5 & 0 &-3\\
12 & -5 & 0 & 3 &-4\\
5 & 0 & -3 & 4 &-3\\
  0 & 3 & -4& 3 & 0\\
 0 &  4&  -3&    0  & 5\\
  0 & 3& 0&   -5  & 12\\
  0 & 0& 5&   -12  &21\\
   0 & 0& 12&   -21 &32 \\
\vdots    & \vdots & \vdots & \vdots & \vdots\\
\end{pmatrix} 
$$
This is not quite the functional we wanted, which had zeros in place of some of the nonzero values. To correct this,
we define  ``truncated'' functionals $ \langle - , \gamma \rangle_{\tau,\kappa}$ by putting zero coefficients in the appropriate spots:
$$
\begin{matrix}
\langle \beta, \gamma \rangle_{\tau,\kappa}&=& \sum_{k\le \kappa} \beta_{\tau,k} \gamma_{\tau,-k}+
\sum_{j <\tau} \sum_k  \beta_{j,k} \gamma_{j,-k} \cr\cr
&&-\sum_{k\le \kappa+1} \beta_{\tau+1,k} \gamma_{\tau,-k}-\sum_{j <\tau} \sum_k \beta_{j+1,k} \gamma_{j,-k}\cr\cr
&&+ \sum_{i>j+1} (-1)^{i-j} \sum_k \beta_{i,k} \gamma_{j,-k}&\cr
\end{matrix}
$$

\begin{theorem}[Positivity 2, \cite{ES1,ES2}]
For $\bF$ the \underline {minimal} free resolution of a finitely generated graded $K[x_0,\ldots,x_m]$-module
and $\mathcal E$ any coherent sheaf on $\mathbb P^m$, we have
$$\langle \beta(\bF), \gamma(\mathcal E) \rangle_{\tau,\kappa} \ge 0.$$
\end{theorem}

\begin{proof} We replace $\sE$ by a general translate as above,
 to achieve homological transversaltity to $\bF$.  Let $E$ be a graded module representing the sheaf $\sE$. Consider the \v Cech resolution $0 \to C^0 \to C^1 \to C^2 \to \ldots$ of $E$ with $C^p= \oplus_{i_0< \ldots,i_p} E[x_{i_0}^{-1}, \ldots,x_{i_p}^{-1}]$ and the tensor product

$$
\begin{matrix}
    &         &  \vdots && \vdots && \vdots & \cr
    &         &  \uparrow && \uparrow && \uparrow & \cr
0 & \lTo &F_0 \otimes C^2  &\lTo &  F_1 \otimes C^2 & \lTo & F_2 \otimes C^2  & \lTo &\ldots \cr
    &         &  \uparrow && \uparrow && \uparrow & \cr
0 & \lTo &F_0 \otimes C^1  &\lTo &  F_1 \otimes C^1 & \lTo & F_2 \otimes C^1  & \lTo &\ldots \cr
    &         &  \uparrow && \uparrow && \uparrow & \cr
0 & \lTo &F_0 \otimes C^0  &\lTo &  F_1 \otimes C^0 & \lTo & F_2 \otimes C^0  & \lTo &\ldots \cr
    &         &  \uparrow && \uparrow && \uparrow & \cr
      &         &  0 && 0 && 0& \cr  
\end{matrix}
$$
of complexes.

By the homological transversality the horizontal cohomology is concentrated in the $F_0$-column. Hence, the total complex has homology  only in non-negative cohomological degrees. The vertical cohomology  in internal degree $0$ on the diagonal or below are the groups 
$$
\begin{matrix}
\; & &\;& &\;& H^2(\widetilde F_2\otimes \sE) \cr
\; & &\;&H^1(\widetilde F_1\otimes \sE) &\;& H^1(\widetilde F_2\otimes \sE) \cr
\; & H^0(\widetilde F_0\otimes \sE)&\;& H^0(\widetilde F_1\otimes \sE)&\;& H^0(\widetilde F_2\otimes \sE) \cr
\end{matrix}
$$
The Euler characteristic of this diagram is again
$\langle \beta(\bF), \gamma(\sE) \rangle$. If we split the internal degree $0$ part of the spectral sequence $H_{vert}(C\otimes F) \Rightarrow H_{tot}(C\otimes F)$ as a sequence of $K$-vector spaces, then we  obtain a complex
 $$\ldots \lTo \;A_p=\bigoplus_{i-j=p} H^j(\widetilde F_i \otimes \sE)\quad \lTo \bigoplus_{i-j=p+1} A_{p+1}=H^j(\widetilde F_{i}\otimes \sE) \; \lTo \ldots$$
 that is exact in negative cohomological degrees. Note that this gives a different proof  (essentially our original proof) of part of the positivity result of Theorem \ref{pos1}, since
  $\langle \beta(\bF), \gamma(\sE) \rangle= \dim \coker(A_1 \to A_0)\ge 0$.  
  
  Consider the submodules
$$B_0 = \bigoplus_{j<\tau} H^j(\widetilde F_j \otimes \sE) \oplus \bigoplus_{k \le \kappa} H^\tau( \sO(-k)^{\beta_{\tau,k}}\otimes \sE)\subset A_0$$
and
$$B_{1} = \bigoplus_{j<\tau} H^j(\widetilde F_{j+1} \otimes \sE) \oplus \bigoplus_{k \le \kappa+1} H^\tau( \sO(-k)^{\beta_{\tau+1,k}}\otimes \sE)\subset A_{1}$$
corresponding  to the truncation.
The diagram
$$
\begin{matrix}
A_0 & \lTo &A_{1} \cr
\uparrow && \uparrow \cr
B_0 & \lTo& B_{1} \cr
\end{matrix}
$$
commutes, because $\bF$ is minimal. Hence,
\begin{eqnarray*}
\langle \beta(\bF), \gamma(\sE)\rangle_{\tau,\kappa}&= &\langle \beta(\bF), \gamma(\sE)\rangle-\dim A_0+\dim B_0+\dim A_1-\dim B_1\cr
&= &\dim \ker(A_1 \to A_0)+ \dim \coker(B_1\to B_0)\cr &&-\dim \ker(B_1 \to B_0) \ge 0,
\end{eqnarray*}
 because $\ker(B_1\to B_0) \subset \ker(A_1 \to A_0)$.
\end{proof}

These stronger versions of the vanishing results of \cite{ES1} allow us to give a direct proof of the
 main theorem of \cite{BS2}:

\begin{proof}[Final part of the Proof of Theorem \ref{bettimain}]
The facet equation, which cuts out the desired face corresponding to a degree sequence $a>b>c$ with $c$ of length $r$ that only differ in positions $\tau$ and $\tau+1\le r$ and satisfies $a_\tau \ge c_{\tau+1}$, is given by taking a supernatural vector bundle
$\sE$  on $\PP^{r-1}\subset \PP^{n-1}$ with root sequence
$(z_1>z_2> \ldots >z_{r-1})=(-b_0>\ldots> -b_{\tau-1} >-b_{\tau+2}>\ldots>-b_r)$, $\kappa=c_{\tau+1}-1 $ and the functional
$$\langle-, \gamma(\sE)\rangle_{\tau,\kappa}.$$ 
Indeed, for a pure module $M$  with a degree sequence $d \le c$, we have
$$\langle \beta(M),\gamma(\sE) \rangle_{\tau,\kappa} =\langle \beta(M),\gamma(\sE) \rangle,$$
and the vanishing follows from Theorem \ref{pos1}: Since the length of $d$ is at least the length of $c$ we can reduce to the case where $M$ has finite length, because the Betti numbers of $M$ and $M/xM$ as an 
$S/xS$-module for a linear nonzero divisor $x$ of $M$ coincide. Furthermore,
$H^{i+1}(\sE(-d_i))=0$ because $-d_i\ge -c_i\ge z_{i+1}$ or $i+1\ge r$. 

The vanishing $\langle \beta(M),\gamma(\sE) \rangle_{\tau,\kappa} =0$ for all pure $M$ with a degree sequence $d \ge a$ is trivially true by our choice of $\sE$ and the truncation.

Finally, $\langle \beta(b), \gamma(\sE) \rangle_{\tau, \kappa} >0$, because $z_{\tau}=-b_{\tau-1}>-b_\tau =-c_{\tau}\ge 
-\kappa=-c_{\tau+1}+1> -c_{\tau+1}>-c_{\tau+2}=-b_{\tau+2}=z_{\tau+1}$ and hence 
$H^\tau(\sE(-b_\tau)) \not=0$. Thus $\langle-, \gamma(\sE)\rangle_{\tau,\kappa}=0$ cuts out the desired face. 
  \end{proof}

Conversely, the essential  facet equations of the cone of cohomology tables of vector bundles are of type $\langle \bF(M), -\rangle_{\tau,\kappa}$ for an appropriate finite length pure module $M$, see \cite{ES1}.

\section{Existence}

To complete the proof of both Boij-S\"oderberg decompositions, it is now enough to establish the existence of supernatural vector bundles and pure resolutions for arbitrary root or degree sequences. In each case there are two methods known. For equivariant resolutions or homogeneous vector bundles in characteristic 0 one can use 
Schur functors \cite{EFW,ES-C,ESW}. For arbitrary fields, one can use  a push down method \cite{ES1}. For bundles this is a simple application of the K\"unneth formula applied to $\mathcal E=\pi_* \mathcal O(a_1,\ldots,a_m)$, where $\pi$ is a finite linear projection 
$\pi: \mathbb P^1 \times \ldots \times \mathbb P^1 \to \mathbb P^m$ and $\mathcal O(a_1,\ldots, a_m)$ is a suitable line bundle on the product.

For resolutions this is an iteration of the Lascoux method \cite{Las} to get the Buchsbaum-Eisenbud family of complexes associated to generic matrices \cite{BE1}: We start with $\mathcal K$, 
a Koszul complex on $ \mathbb P^{r-1}\times \mathbb P^{m_1} \times \ldots \times \mathbb P^{m_s} $ of  $r+\sum_{i=1}^s m_i$ forms of multidegree $(1,\ldots,1)$ tensored with  $\mathcal O(-d_0,a_1,\ldots,a_s)$. Here $s$ is the number of desired non-linear maps and $m_j+1$ is the desired degree of the $j$-th non-linear map. 
The spectral sequence for $R\pi_* \mathcal K$ of the projection $ \pi: \mathbb P^{r-1}\times \mathbb P^{m_1} \times \ldots \times \mathbb P^{m_s} \to \mathbb P ^{r-1}$ gives rise to the desired complex if we choose $a_1,\ldots ,a_s$
suitably. Indeed, we may apply Proposition \ref{pushdown construction} $s$-times: 
For any product $X_1 \times X_2$ with projections $p:X_1 \times X_2 \to X_1$ and $q:X_1 \times X_2 \to X_2$ and sheaves $\sL_i$ on $X_i$, we set
$$
\sL_1 \boxtimes \sL_2 := p^*\sL_1 \otimes q^* \sL_2.
$$

\begin{prop}
\label{pushdown construction} 
Let $\sF$ be a sheaf on $X \times \PP^m$, and let $p:X\times \PP^m \to X$
be the projection. Suppose that $\sF$ has a resolution of the form
$$
0 \to \sG_N \boxtimes \sO(-e_N) \to \cdots \to  \sG_0 \boxtimes \sO(-e_0) \to \sF \to 0
$$
with degrees $e_0< \cdots < e_N$. If this sequence
 contains the subsequence $(e_{k+1},\ldots,e_{k+m}) \\=(1,2,\ldots, m)$ for some $k\ge 0$ then 
$$
R^\ell p_*\sF=0 \hbox{ for } \ell>0
$$
and $p_* \sF$ has a resolution on $X$ of the form
\begin{align}
 0 \to \sG_N \otimes H^m\sO(-e_N) &\to \cdots  \cr
  \to \sG_{k+m+1}& \otimes H^m \sO(-e_{k+m+1} ) \stackrel{\phi}{\to}  \sG_{k} \otimes H^0 \sO(-e_{k}) \to \cr
&\qquad\qquad\qquad\qquad\qquad\qquad
 \cdots \to \sG_0 \otimes H^0 \sO(-e_0)
\end{align}
\end{prop}

\begin{proof} From the numerical hypotheses we see that $e_i\leq 0$ for $i\leq k$
and $e_i\geq m+1$ for $i\geq k+m+1$.
Consider the spectral sequence
$$
E^{i,-j}_1=R^ip_* (\sG_j\boxtimes \sO(-e_j)) \Rightarrow R^{i-j} p_* \sF.
$$
By the projection formula, the terms of the $E_1$ page are
$$
R^ip_* (\sG_j\boxtimes \sO(-e_j))=
\begin{cases}
\sG_j\otimes H^m(\PP^m,\sO(-e_j)) &\hbox{ if } j\geq k+m+1 \hbox{ and } i=m \cr
\sG_j \otimes H^0(\PP^m, \sO(-e_j)) &\hbox{ if } j\le k \hbox{ and } i=0 \cr
0 & \hbox{ otherwise. }
\end{cases}
$$
Thus, the spectral sequence degenerates to the complex (1), where
$\phi$ is a differential from the $m$-th page 
and the other maps are differentials from the first page.
In particular, only terms  $E^{i,-j}_{\infty}$ with  $i\leq j$ can be nonzero. 
On the other hand, the terms $R^{i-j} p_* \sF$  can be nonzero only for  $i\ge j$. 
Hence, the complex (1) is exact and resolves 
$\oplus_{i\ge 0} E^{i,-i}_\infty=E^{0,0}_\infty= p_* \sF$, 
while the higher direct images of $\sF$ vanish.
\end{proof}

\section{Applications, Extensions of the Basic Theory and Open Questions}

The application that motivated Boij and S\"oderberg to make their Conjecture, was the following sharp
version of the Multiplicity Conjecture of Huneke and Srinivasan \cite{HS}.

\begin{theorem}[\cite{BS2}] For any finitely generated module $M$ of projective dimension $r$ and codimension $s$ generated in degree $0$, we have the following bounds for the Hilbert series:
$$
(\prod_{i=1}^r a_i)\,{H ( \beta(a), t) }
\le
\frac{H (M, t)}{\beta_{0,0}(M)}
\le 
(\prod_{i=1}^s b_i)\,H ( \beta(b), t) \; ,
$$
where $a=(0,a_1 , a_2 , . . . , a_r)$ are the minimal shifts and $b=(0, b_1 , b_2 , . . . , b_s)$ are the maximal shifts in 
a minimal free resolution of $M$. Equality on either side implies that the resolution is pure. 
In particular, the right hand inequality implies the Multiplicity Conjecture, that is, the multiplicity of $M$ is bounded by
$$\mult(M )
 \le \beta_{0,0} (M )\, \frac{b_1\cdot \ldots \cdot b_s}{s!}
$$
with equality if and only if $M$ is Cohen-Macaulay with a pure resolution. 

\end{theorem}

\begin{proof}[Sketch] If $ a$  is a degree sequence of length $r$ with $a_0=0$, then the Betti table $(\prod_{i=1}^r a_i)\,\beta(a)$ is normalized such that $\beta_{0,0}=1$.
Given two degree sequences $a <b$ with $a_0=b_0=0$, then
the Hilbert series of the normalized tables satisfy
$$H((\prod_{i=1}^r a_i) {\beta(a),t)}< {H((\prod_{i=1}^s b_i)\beta(b),t)}.$$
The result follows because the normalized Boij-S\"oderberg decomposition  is a convex combination.
\end{proof}

Let $a < b$ be two degree sequences of equal length.
The part of the Boij-S\"oderberg cone of tables $\beta=(\beta_{i,j})$ with
$\beta_{i,j}=0 \hbox{ unless }a_i \le j \le b_i$ is a finite,  equi-dimensional simplical fan.

Turning to the monoid of Betti tables of modules, we have:

\begin{theorem}[Erman, \cite{Erman1}] The monoid of Betti tables of Cohen-Macaulay modules with Betti tables bounded by the degree sequences $a < b$ is finitely generated.
\end{theorem}

Note that the index of actual Betti tables among the integral points on a ray of the Boij-S\"oderberg cone  can be arbitrary large \cite{EES}. However, along the extremal rays, Eisenbud and Weyman conjecture that  any sufficiently large integral point is the Betti table of a module. 

To understand the monoid is substantially more difficult than understanding the cone. For example, unlike the Boij-S\"oderberg cone the monoid of Betti tables depends on the characteristic of the ground field \cite{Kunte}.
A case where we understand the monoid completely can be found in \cite{EES}.

We believe that the most important next step in trying to understand the monoid better, would be a proof of the Eisenbud-Buchsbaum-Horrocks rank conjecture \cite{BE2}. Daniel
Erman \cite{Erman2} uses  the Boij-S\"oderberg decomposition
to prove  the rank conjecture for $M=S/I$ a cyclic module, provided that  the minimal generators of the ideal $I$ are sufficiently large compared to the regularity.

Turning to coherent sheaves, it is no longer true that the cohomology table of any sheaf is a finite sum of tables of supernatural sheaves of various dimensions. What remains true is that a cohomology table of an arbitrary coherent sheaf on $\PP^m$ is an infinite (convergent) sum with non-negative coefficients of cohomology tables of supernatural sheaves, whose zero sequences form a chain, see \cite{ES2}. This result, however, does not characterize the Boij-S\"oderberg cone of coherent sheaves on $\PP^m$ but only its closure in
$\bigoplus_{i=0}^m \prod_{j\in\ZZ} \RR$.

If $X \subset \PP^m$ is a subvariety of dimension $d$  we can ask about the Boij-S\"oderberg cone of its coherent sheaves. Consider a linear Noether normalization $\pi: X \to \PP^d$. Then, for any sheaf $\mathcal F$ on $X$, the cohomology table of $\mathcal F$ and $\pi_*\sF$ coincide. Hence the Boij-S\"oderberg cone
of $(X,\sO_X(1))$ is a subcone of the one of $(\PP^d,\sO(1))$. If they coincide then there exists a sheaf $\sU$ on $X$, whose cohomology table coincides with that of $\sO_{\PP^d}$ up to a multiple. By Horrocks criterion \cite{Ho}, this implies $\pi_* \sU \cong \sO_{\PP^d}^r$. By the very definition \cite{ES-C}, this means that $\mathcal U$ is an Ulrich sheaf on $X$. 

Conversely, if an Ulrich sheaf exists then for a sheaf $\sG$ on $\PP^d$, the cohomology table of
$\sU \otimes \pi^* \sG$ and $\sG$ coincide up to the factor $r$. This proves

\begin{theorem}[\cite{ES-V}] The Boij-S\"oderberg cone of the coherent sheaves on a variety $X$ of dimension $d$
with respect to a very ample polarization $\sO_X(1)$ coincides with the Boij-S\"oderberg cone of $(\PP^d,\sO(1))$ if and only if $X$ carries an Ulrich sheaf.
\end{theorem}

Varieties that have an Ulrich sheaf include curves and hypersurfaces. They are closed under Segre products, Veronese re-embeddings and transversal intersections. In \cite{ES-C} we conjecture that every variety has an Ulrich sheaf.

Very little is known for the extension of this theory to the multi-graded setting. I believe that there will be beautiful results ahead in this direction.


\vbox{\noindent Author Addresses:\par
\smallskip
\noindent{David Eisenbud}\par
\noindent{Department of Mathematics, University of California, Berkeley,
Berkeley CA 94720}\par
\noindent{eisenbud@math.berkeley.edu}\par
\smallskip
\noindent{Frank-Olaf Schreyer}\par
\noindent{Mathematik und Informatik, Universit\"at des Saarlandes, Campus E2 4, 
D-66123 Saarbr\"ucken, Germany}\par
\noindent{schreyer@math.uni-sb.de}\par
}

\end{document}